\documentclass{article}

\usepackage{amsmath}
\usepackage{amsfonts}
\usepackage{amssymb}
\usepackage{amsthm}
\usepackage{graphicx}
\usepackage{leftidx}
\usepackage[numbers]{natbib}

\newtheorem{theorem}{Theorem}
\newtheorem{corollary}[theorem]{Corollary}

\newtheorem{proposition}[theorem]{Proposition}
\newtheorem{definition}[theorem]{Definition}

\newcommand{\ie}{\emph{i.e.}}
\newcommand{\emphb}[1]{\emph{\bf #1}}

\newcommand{\pars}[1]{\ensuremath{\left(#1\right)}}
\newcommand{\brks}[1]{\ensuremath{\left[#1\right]}}
\newcommand{\brcs}[1]{\ensuremath{\left\{#1\right\}}}

\newcommand{\brcscond}[2]{\ensuremath{\left\{#1\vphantom{#2}\ \right|\left.\vphantom{#1}#2\right\}}}

\newcommand{\integer}{\ensuremath{\mathbb{Z}}}
\renewcommand{\natural}{\ensuremath{\mathbb{N}}}

\newcommand{\real}{\ensuremath{\mathbb{R}}}

\renewcommand{\min}{\ensuremath{\operatornamewithlimits{min}}}

\newcommand{\Image}{\ensuremath{\operatorname{image}}}

\newcommand{\Rank}{\ensuremath{\operatorname{rank}}}

\newcommand{\supp}{\ensuremath{\operatorname{supp}}}

\newcommand{\Mod}{\ensuremath{\operatorname{\bf Mod}}}
\newcommand{\GrdMod}{\ensuremath{\operatorname{\bf GrdMod}}}
\newcommand{\PersMod}{\ensuremath{\operatorname{\bf PersMod}}}
\newcommand{\Region}{\ensuremath{\operatorname{Region}}}

\newcommand{\Ann}{\ensuremath{\operatorname{Ann}}}

\newcommand{\mcA}{\ensuremath{\mathcal A}}
\newcommand{\mcB}{\ensuremath{\mathcal B}}
\newcommand{\mcC}{\ensuremath{\mathcal C}}
\newcommand{\mcD}{\ensuremath{\mathcal D}}

\newcommand{\mcF}{\ensuremath{\mathcal F}}

\newcommand{\mcH}{\ensuremath{\mathcal H}}
\newcommand{\mcI}{\ensuremath{\mathcal I}}

\newcommand{\mcO}{\ensuremath{\mathcal O}}

\begin{document}

\title{Exterior Critical Series of Persistence Modules}
\date{}
\author{Pawin Vongmasa \and Gunnar Carlsson}
\maketitle

\begin{abstract}{\bf
The persistence barcode
	is a well-established complete discrete invariant for
	finitely generated persistence modules
	\cite{ComputingPersistentHomology} \cite{ZigzagPersistence}.
Its definition, however, does not extend
	to multi-dimensional persistence modules.
In this paper, we introduce a new discrete invariant:
	the \emph{exterior critical series}.
This invariant is complete in the one-dimensional case
	and can be defined for multi-dimensional persistence modules,
	like the \emph{rank invariant}
	\cite{TheoryOfMultidimensionalPersistence}.
However,
	the exterior critical series can detect some features that are
	not captured by the rank invariant.
}\end{abstract}

\section{Introduction}
The structure theorem for finitely generated modules
	over a graded principal ideal domain
	is a well-known result that applies almost directly to
	finitely generated persistence modules with integer grading
	\cite{ComputingPersistentHomology}.
This defines the \emphb{persistence barcode} as an invariant
	of the corresponding persistence module.
The persistence barcode is simply a collection of \emphb{bars}
	(synonymously, \emphb{intervals}),
	each of which can be represented by two numbers,
	its \emphb{birth time} and \emphb{death time}.

As the concept of persistence modules is extended to multiple dimensions,
	the structure theorem does not extend.
The notion of bars becomes undefined.
In \cite{TheoryOfMultidimensionalPersistence},
	Carlsson and Zomorodian
	studied the structure of multi-dimensional persistence modules,
	showed that no complete discrete invariant
\footnote{
	A \emphb{discrete invariant}, as discussed in 
		\cite{TheoryOfMultidimensionalPersistence},
	is an invariant that does not depend on the ground field.
	The ground field will be denoted by $R$ throughout this paper.
}
	can exist,
	and introduced three discrete invariants
	that can be defined for multi-dimensional persistence modules:
	$\xi_0$, $\xi_1$ and the rank invariant.
$\xi_0$ and $\xi_1$ are simply collections (multisets) of
	birth times and death times.
They are less informative than the barcode
	in the one-dimensional case because
	pairing information of birth times and death times
	is not available.
The rank invariant, on the other hand,
	is complete in the one-dimensional case.

In this paper, we define a new discrete invariant
	called the \emphb{exterior critical series},
	a special case of a more general framework also presented here.
Its definition is based on the \emphb{critical series},
	which is (for 1-dimensional persistence modules)
	essentially the same as $\xi_0$ and $\xi_1$,
	the multiset of birth times and the multiset of death times,
	respectively.
(See Proposition \ref{proposition:finitelypresentedmatching}.)
The main idea is that although pairing information is
	not maintained by $\xi_0$ and $\xi_1$,
	it is in some ways encoded in $\xi_0$ and $\xi_1$ of
	higher exterior powers of the module.
More specifically, suppose $M$ is a persistence module.
The exterior critical series of $M$ is the collection
	of critical series of $M, \Lambda^2 M, \Lambda^3 M$, and so on,
	where $\Lambda^p M$ is the $p$-th exterior power of $M$.
We show in Theorem \ref{theorem:main} that
	the exterior critical series is complete, \ie,
	as informative as the barcode,
	for $1$-dimensional persistence modules.
One conclusion that can be drawn is that
	the collection of $\xi_0(\Lambda^p M)$ and $\xi_1(\Lambda^p M)$
	does carry information about pairing of birth and death times in $M$.
It is therefore reasonable to use this collection as an invariant
	even when $M$ is a multi-dimensional persistence module.
In this paper, we give a natural definition of the exterior critical series
	that easily extends to multi-dimensional persistence modules.
The extension involves new invariants $\xi_2, \xi_3, \ldots$,
	which are simply generalization of $\xi_0$ and $\xi_1$.
All $\xi_i$ can be defined via a free resolution of the module and
	an additive functor.
(See Section \ref{section:invariantswithrespecttofunctors}.)

The rank invariant, as defined in 
	\cite{TheoryOfMultidimensionalPersistence},
	is also a discrete invariant that is complete in the $1$-dimensional case
	and extends to the multi-dimensional case,
	just like the exterior critical series.
However, there are examples of modules that can be distinguished
	by their exterior critical series but not by their rank invariants.
We give one example in Fig. \ref{figure:invariantexample}.

The organization of this paper is as follows.
We start with basic definitions of
	\emph{graded modules} and
	their tensor products.
\emph{Persistence modules} are then defined as
	graded modules with more information added.
A procedure to derive an invariant from
	an additive functor is discussed
	in Section \ref{section:invariantswithrespecttofunctors},
	and the \emph{critical series} is simply the result of the said procedure.
Section \ref{section:structuretheorems} follows
	with a list of proven structure theorems for
	(certain classes of) 1-dimensional
	persistence modules, phrased in our terminology.
We discuss \emph{causality}
	and state the specialized classification result
	for ``nice'' \emph{causal persistence modules} in
	Corollary \ref{corollary:causalbarcodedecomposition}.
Section \ref{section:main} starts with
	a list of definitions specific to the class of
	persistence modules of interest
	($1$-dimensional, finitely-presented, and bounded).
Another definition of the critical series is given,
	and Proposition \ref{proposition:finitelypresentedmatching}
	shows that in this specific setting,
	it coincides with the more general definition given earlier
	in Section \ref{section:invariantswithrespecttofunctors}.
Finally, Theorem \ref{theorem:main}, the main result,
	is stated and proved in Section \ref{section:mainresult}.

Definitions of terms in this paper are made quite general
	for the purpose of future extension.
In the statement of the main result,
	we choose to work with real-valued grading instead of the traditional
	integer-valued grading.
We choose so because in most applications,
	the persistence module is constructed from
	a filtration with a real-valued parameter,
	and in most works, if not all,
	the notion of \emph{stability} of the invariant
	relies on the metric on the parameter space
	rather than the integer grades.
Also, if one desires to go back to integer grades,
	one can simply consider the integers as a subset of the real numbers.

\section{Notations, Conventions and Background}
We will assume $R$ is a commutative ring with unity throughout the paper.

\subsection{Graded $R$-Modules}
Suppose $G$ is a set.
A \emphb{$G$-graded $R$-module} is an $R$-module $M$
	together with a collection $\brcs{M_g}_{g\in G}$ of $R$-submodules of $M$
	indexed by $G$ such that $M = \bigoplus_{g \in G} M_g$.
This indexed collection is called \emphb{grading of $M$},
	and $G$ is the set of \emphb{grades}.

The indexed collection is in fact a function from $G$ to the collection of 
	submodules, so
	a $G$-graded $R$-module is determined by
	the couple $(M, \Gamma)$ where $\Gamma:g \mapsto M_g$
	satisfies $M = \bigoplus_{g \in G} \Gamma(g)$.
However, we will usually avoid the reference to $\Gamma$
	and assume $\Gamma$ is given when we say that
	``$M$ is a $G$-graded $R$-module''
	(instead of ``$(M, \Gamma)$ is a $G$-graded $R$-module'').
The notation $M_g$ will also be assumed present,
	and it is equal to $\Gamma(g)$.
(We need $\Gamma$ because it matters
	which $g \in G$ goes to which submodule of $M$.)

Every $G$-graded $R$-module $M$ comes equipped with natural projections
	$\pi_g:M \to M_g$ and inclusions $\iota_g:M_g \to M$.
The \emphb{support of $M$} is
	$\supp(M) = \brcscond{g \in G}{M_g \ne 0}$.
The \emphb{support of $m \in M$} is
	$\supp(m) = \brcscond{g \in G}{\pi_g(m) \ne 0}$.
We say that $M$ or $m \in M$ is \emphb{bounded} if its support
	is contained in some interval $[g, g'] \subseteq G$
\footnote{
	An interval $[g, g']$ is defined by
	$[g, g'] = \brcscond{h}{g \preceq h \preceq g'} \subseteq G$.
	Intervals of forms $[g, g')$, $(g, g']$ and $(g, g')$ are defined similarly.
}.
An element $m \in M$ is \emphb{homogeneous} if
	$\supp(m)$ has cardinality $0$ or $1$, \ie,
	there exists $g \in G$ such that $m \in M_g$.
If $m \in M$ is non-zero and homogeneous,
	then there exists a \emph{unique} $g \in G$ with $m \in M_g$.
$g$ is called the \emphb{degree of $m$},
	written $\deg(m)$.

An $R$-module homomorphism $\varphi:M \to M'$
	between $G$-graded modules
	is \emphb{graded} if $\varphi(M_g) \subseteq M'_g$
	for all $g \in G$.
We denote by $\GrdMod(R, G)$
	the category whose objects are $G$-graded $R$-modules
	and morphisms are graded $R$-module homomorphisms.
Submodules are defined by monomorphisms
	and quotient modules are defined by epimorphisms.
It follows that every
	$G$-graded $R$-submodule is generated by homogeneous elements.

The \emphb{tensor product} between graded modules are defined as follows.
Suppose $H, H'$ and $G$ are sets, $\cdot: H \times H' \to G$
	is a binary function,
	$M \in \GrdMod(R, H)$ and $M' \in \GrdMod(R, H')$.
The usual tensor product $M \otimes_R M'$ can be given
	the structure of a $G$-graded $R$-module by defining the grading
\[
	\pars{M \otimes_R M'}_g =
	\bigoplus_{\substack{(h, h') \in H \times H'\\ h \cdot h' = g}}
	M_h \otimes_R M'_{h'}.
\]
In other words, the binary operation
	$\cdot: H \times H' \to G$ induces the bifunctor
\[
	\otimes_R: \GrdMod(R, H) \times \GrdMod(R, H') \to
	\GrdMod(R, G).
\]
More generally,
	suppose we have a collection
	of modules $M_i \in \GrdMod(R, H_i)$, $i \in \mcI$.
A function $o: \prod_{i\in\mcI} H_i \to G$
	makes the tensor product $\bigotimes_{i\in\mcI}M_i$
	into a $G$-graded $R$-module.
The induced tensor product is a multifunctor
	from $\prod_{i\in\mcI} \GrdMod(R, H_i)$ to $\GrdMod(R, G)$.

\subsection{Persistence Modules}
Suppose $G$ is a partially ordered set and
	$X$ is a commutative monoid acting on $G$.
We say that $(X, G)$ is a \emphb{persistence grading} if
\begin{enumerate}
	\item \label{actionpositivity}
		For $g \in G$ and $x \in X$, $g \preceq xg$.
	\item \label{actiontransitivity}
		For $g, g' \in G$ with $g \preceq g'$,
			there exists a unique $x \in X$ such that $xg = g'$.
\end{enumerate}
These conditions imply that $X$ is a cancellative monoid
	that can be embedded as a subset of $G$.
Any choice of embedding induces the same order on $X$,
	so we assume $X$ is ordered by $\preceq$ also.
By condition \eqref{actiontransitivity}
	on the action of $X$ on $G$,
	we define $g/g' =\frac g{g'} = x$,
	where $g, g' \in G$, $g \succeq g'$
	and $g=xg'$.

A $G$-graded $R$-module $M$ can be made into a \emphb{persistence module}
	by adding the following:
\begin{enumerate}
\item \label{persistencegrading}
  A persistence grading $(X, G)$.
	$X$ is called the \emphb{monomial monoid}
		and its elements are called \emphb{monomials}.
\item \label{compatibleaction}
	An $R[X]$-module structure on $M$ where
		such that
		$xM_g \subseteq M_{xg}$ for all $x \in X$ and $g \in G$.
	$P = R[X]$ is called the \emphb{ring of polynomials}, and
		its elements are called \emphb{polynomials}.
\end{enumerate}
For each $x \in X \subseteq P$, the action of $x$ on $M$
	induces an $R$-module endomorphism $M^x: M \to M$.
When restricted to $M_g$, $g \in G$, the image of $M^x$
	is contained in $M_{xg}$ (by \eqref{compatibleaction}),
	so we define $M^x_g: M_g \to M_{xg}$
	by $M^x_g(m) = M^x(m) = xm$.
$M^x_g$ is obviously an $R$-module homomorphism.
We call $M^x_g$ the \emphb{monomial homomorphism of $x$ on $M_g$}.

The set of \emphb{$X$-annihilators of $m \in M$}
	is $\Ann_X(m) = \brcscond{x \in X}{xm = 0}$.
This is a subsemigroup of $X$.
The \emphb{lifespan of $m \in M$}, written $L(m)$, is the complement of
	$\Ann_X(m)$, \ie, $L(m) = X - \Ann_X(m)$.
Note that the set of annihilators of $m$,
	defined conventionally by $\Ann(m) = \brcscond{p \in P}{pm = 0}$
	may contain more information than $\Ann_X(m)$ only when $R$ is not a field.
When $R$ is a field, $\Ann(m)$ is simply the vector space with basis
	$\Ann_X(m)$.

With $R, X$ and $G$ fixed,
	we define the category $\PersMod(R,X,G)$ whose
	objects are persistence modules with respect to the given $R, X$ and $G$,
	and whose morphisms are \emph{$G$-graded} $P$-module homomorphisms.
Persistence submodules and quotient modules are defined by
	monomorphisms and epimorphisms respectively.
We simply say that ``$M$ is a persistence module''
	when it is clear what $R$, $X$ and $G$ are,
	and morphisms of $\PersMod(R,X,G)$ will be called
	``persistence module homomorphisms''.

A persistence module can easily be made into a graded $R$-module
	by forgetting the action of $X$ on the module.
Conversely, a graded $R$-module can be made into a persistence module
	by adding the trivial action of $X$ on the module.
The former operation is a forgetful functor from $\PersMod(R, X, G)$
	to $\GrdMod(R, G)$;
	the latter operation is an embedding of $\GrdMod(R, G)$
	in $\PersMod(R, X, G)$.

\subsubsection*{Tensor Products of Persistence Modules}
Suppose we have
	$(X, G)$ a persistence grading,
	$M \in \PersMod(R, X, H)$,
	$M' \in \PersMod(R, X, H')$,
	and $\cdot: H \times H' \to G$ a binary function such that
	$(xh) \cdot h' = h \cdot (xh') = x(h \cdot h')$ for all $x \in X$,
	$h \in H$ and $h' \in H'$.
The tensor product $M \otimes_P M'$ can be given the structure
	of a persistence module in $\PersMod(R, X, G)$ as follows.

Let
	$\iota: M \otimes_R M' \to M \otimes_P M'$ be
	a surjective $R$-module homomorphism defined by
	$\iota\pars{m \otimes_R m'} = m \otimes_P m'$.
The kernel of $\iota$ is generated by
	elements of the form $xm \otimes m' - m\otimes xm'$
	for $m \in M$, $m' \in M'$ and $x \in X$.
Passing from $M \otimes_R M'$ to $M \otimes_P M'$ via $\iota$
	is simply extension of scalars from $R$ to $P$.
We impose the following $G$-grading on $M \otimes_P M'$:
\begin{equation}
	\pars{M \otimes_P M'}_g 
	= \iota\pars{\pars{M \otimes_R M'}_g}
	= \bigoplus_{\substack{(h, h') \in H\times H'\\h\cdot h' = g}}
	\iota\pars{M_h \otimes_R M'_{h'}}.
	\label{eq:persistencetensorgrading}
\end{equation}
To check that the action of $P$ on $M \otimes_P M'$ indeed satisfies
	the condition
	$x\pars{M \otimes_P M'}_g \subseteq \pars{M \otimes_P M'}_{xg}$
	for all $x \in X$ and $g \in G$,
	suppose $h \in H$, $h' \in H'$, $m \in M_h$, $m' \in M'_{h'}$ and $x \in X$
	are given.
Then
\[
	x\iota\pars{m \otimes_R m'}
	= x\pars{m \otimes_P m'}
	= xm \otimes_P m' = \iota\pars{(xm) \otimes_R m'}
	\in \iota\pars{M_{xh} \otimes_R M'_{h'}}
	\subseteq \pars{M \otimes_P M'}_{xh\cdot h'}.
\]
Hence, \eqref{eq:persistencetensorgrading}
	defines $M \otimes_P M'$ as a persistence module.
We conclude that the binary function $\cdot$ induces the bifunctor
\[
	\otimes_P: \PersMod(R, X, H) \times \PersMod(R, X, H') \to
	\PersMod(R, X, G).
\]

More generally, suppose there is a collection of modules
	$M_i \in \PersMod(R, X, H_i)$, $i \in \mcI$,
	and a function $o: \prod_{i\in\mcI} H_i \to G$ such that
	$x o(\gamma) = o(\gamma_i)$ for all $\gamma \in \prod_{j\in\mcI} H_j$
	and $i \in \mcI$,
	where $\gamma_i$ are defined by
\[
	\gamma_i(j) = \begin{cases}
		x \gamma(j) & ; i = j \\
		\gamma(j) & ; i \ne j
	\end{cases}
\]
Then the function $o$ induces a tensor product multifunctor
	from $\prod_{i\in\mcI} \PersMod(R, X, H_i)$ to
	$\PersMod(R, X, G)$.
If all $H_i$ are monoids,
	we only need $o: \bigoplus_{j\in\mcI} H_j \to G$.

Suppose $m_i \in M_i$ for $i \in \mcI$.
Then $\Ann_X\pars{\bigotimes_{i\in\mcI} m_i} =
		\bigcup_{i\in\mcI} \Ann_X\pars{m_i}$.
In terms of lifespans, we have
	$L\pars{\bigotimes_{i\in\mcI} m_i} = \bigcap_{i\in\mcI} L(m_i)$.
	
\subsubsection*{Tensor Products in $\PersMod(R, X, G)$}
	\label{section:tensorpowers}
In order for the tensor product to be defined
	within one category $\PersMod(R, X, G)$,
	we need the binary function $\cdot:G \times G \to G$
	such that $x(g \cdot g') = (xg) \cdot g' = g \cdot (xg')$
	for $g, g' \in G$ and $x \in X$.
Instead of requiring separate functions for
	tensor products with different numbers of factors,
	it is more convenient to impose another condition: that $\cdot$ is associative.
Now $G$ and $\cdot$ form a semigroup.
We will omit writing $\cdot$ in an expression
	involving elements of $X$ and $G$.
(This semigroup structure does not define
	a tensor product with infinitely many factors,
	but we will not need such generality.)

\subsubsection*{Tensor Powers, Symmetric Powers and Exterior Powers}
Consider the $n$-th symmetric power of a persistence module $M$,
	written $S^n(M)$.
$S^n(M) = M^{\otimes n} / K$, where $K$ is the module generated by
	symmetry.
$K$ is a persistence module because
	it can be generated by homogeneous elements, so
	$S^n(M)$ is a persistence module.

Similarly, the $n$-th exterior power of $M$ is
	$\Lambda^n(M) = M^{\otimes n} / K$, where $K$ is the module
	generated by antisymmetry.
$K$ is, again, a persistence module, so
	$\Lambda^n(M)$ is also a persistence module.

In $M^{\otimes n}$, $S^n(M)$ and $\Lambda^n(M)$, we have the
	same relation regarding $\Ann_X$ and $L$ of simple tensors
	and their components: for $m_i \in M$, if
	$m_i \ne 0$, then
\[
	\begin{array}{c}
		\displaystyle{
		\Ann_X\pars{\bigotimes_{i=1}^n m_i} =
		\Ann_X\pars{\bigodot_{i=1}^n m_i} =
		\bigcup_{i=1}^n \Ann_X(m_i)
		} \\
		\displaystyle{
		L\pars{\bigotimes_{i=1}^n m_i} =
		L\pars{\bigodot_{i=1}^n m_i} =
		\bigcap_{i=1}^n L(m_i),
		}
	\end{array}
\]
where $\bigodot$ stands for the symmetric product.
Moreover, if $m_1 \wedge \ldots \wedge m_n \ne 0$, we also have
\[
	\begin{array}{c}
		\displaystyle{
		\Ann_X\pars{\bigwedge_{i=1}^n m_i} =
		\bigcup_{i=1}^n \Ann_X(m_i)
		} \\
		\displaystyle{
		L\pars{\bigwedge_{i=1}^n m_i} =
		\bigcap_{i=1}^n L(m_i),
		}
	\end{array}
\]
where $\bigwedge$ stands for the exterior product.

\subsection{Invariants With Respect to Functors}
\label{section:invariantswithrespecttofunctors}

Suppose $\mcF$ is an additive covariant functor
	from $\PersMod(R, X, G)$ to $R$-$\Mod$,
	and suppose $M$ is a persistence module with a free resolution
\[
	\ldots \to F_2 \to F_1 \to F_0 \to M \to 0
\]
	with $F_i$ free persistence modules.
The \emphb{$\mcF$-homology sequence of $M$}
	is the homology sequence of the chain complex
\[
	\ldots \to \mcF(F_2) \to \mcF(F_1) \to \mcF(F_0) \to 0.
\]
We will denote by $H\mcF_i(M)$ the
	$i$-th homology module,
	and write $H\mcF_*(M)$ to refer to the whole sequence
	$(H\mcF_0(M), H\mcF_1(M), \ldots)$.
The sequence $H\mcF_*(M)$ is independent
	of the choice of resolution because $\mcF$ is additive.

If each $H\mcF_i(M)$ has finite rank,
	we say that $M$ \emphb{admits an $\mcF$-Hilbert-Poincar\'e series},
	and we define
	the \emphb{$\mcF$-Hilbert-Poincar\'e series of $M$} by
\begin{equation}
	\mcH\mcF(M) = \sum_{i=0}^\infty \Rank\pars{H\mcF_i(M)}t^i
	\label{eq:hilbertpoincareseries}
\end{equation}
where $t$ is an indeterminate.
$\mcH\mcF(M)$ is a member of $\integer[[t]]$,
	the ring of formal power series with integer coefficients.

In addition to all $H\mcF_i(M)$ having finite rank,
	if only finitely many $H\mcF_i(M)$ are non-zero modules,
	we say that $M$ \emphb{admits an $\mcF$-Euler characteristic},
	and we define
	the \emphb{$\mcF$-Euler characteristic of $M$} by
	substituting $t = -1$ in the $\mcF$-Hilbert-Poincar\'e series:
\begin{equation}
	\chi\mcF(M) = \sum_{i=0}^\infty \Rank\pars{H\mcF_i(M)}(-1)^i
	= \mcH\mcF(M)|_{t = -1}.
	\label{eq:eulercharacteristic}
\end{equation}

By construction, $H\mcF_*$ is an invariant
	for persistence modules.
$\mcH\mcF$ is an invariant for the subclass of
	persistence modules that admit $\mcF$-Hilbert-Poincar\'e series.
$\chi\mcF$ is an invariant for the subclass of
	persistence modules that admit $\mcF$-Euler characteristics.

\subsection*{Causal Onset Functor and Critical Series}
Suppose $M$ is a persistence module.
For a fixed $g \in G$, define the \emphb{lower sum of $M$ at $g$} by
\[
	M_{\prec g} = \sum_{h \prec g} \Image\pars{M_h^{g/h}} 
	= \sum_{h \prec g} (g/h)M_h \subseteq M_g.
\]
(Recall that $g/h$ is defined as the unique $x \in X$ such that $xh = g$.)
The \emphb{causal onset functor} $\mcO_g$ is defined by
\[
	\mcO_g(M) = M_g/M_{\prec g}.
\]
$\mcO_g$ maps an object in $\PersMod(R, X, G)$ to an object in $R$-$\Mod$.
Naturally, morphisms should be mapped by
\begin{equation}
	\pars{\mcO_g(\varphi)}(m + M_{\prec g}) = \varphi(m) + M'_{\prec g}
	; \qquad m \in M_h, h \prec g
	\label{eq:lowerendfunctoronmorphisms}
\end{equation}
where $\varphi:M \to M'$ is a persistence module homomorphism.
\begin{proposition}
	$\mcO_g$ is an additive functor.
\end{proposition}
\begin{proof}
	This follows immediately from \eqref{eq:lowerendfunctoronmorphisms}.
	What may not be so obvious is that
		$\mcO_g$ is indeed well-defined by
		\eqref{eq:lowerendfunctoronmorphisms}.
	Suppose $\varphi:M \to M'$ is a persistence module homomorphism,
		and $m, n \in M_g$ with $m - n \in M_{\prec g}$.
	Then $\pars{\mcO_g(\varphi)}(m + M_{\prec g}) -
	\pars{\mcO_g(\varphi)}(n + M_{\prec g}) = \varphi(m - n) + M'_{\prec g}$.
	We need $\varphi(m - n) \subseteq M'_{\prec g}$, \ie,
	$\varphi(M_{\prec g}) \subseteq M'_{\prec g}$, but this follows from
	$\varphi(xM_h) = x\varphi(M_h) \subseteq xM'_h$ for any $h \in G$.
	Therefore, $\mcO_g(\varphi)$ is well-defined.
\end{proof}
\begin{definition}
A persistence module $M$ \emphb{admits a causal series} if
	it admits the $\mcO_g$-Hilbert-Poincar\'e series for each $g \in G$.
The \emphb{causal series of $M$} is the function
	$\mcH_{\mcO}(M): G \to \integer[[t]]$ defined by
\begin{equation}
	\mcH_{\mcO}(M)(g) = \mcH\mcO_g(M).
	\label{eq:causalseries}
\end{equation}
$M$ \emphb{admits a critical series} if
	it admits the $\mcO_g$-Euler characteristic for each $g \in G$.
The \emphb{critical series of $M$} is the function
	$\chi_\mcO(M): G \to \integer$ defined by
\begin{equation}
	\chi_{\mcO}(M)(g) = \chi\mcO_g(M).
	\label{eq:criticalseries}
\end{equation}
\end{definition}
The invariant of interest in this paper
	is the \emphb{exterior critical series}, which is
	the collection of critical series of the module and
	its exterior powers.
\begin{definition}
The \emphb{exterior critical series of $M$} is the function
	$\chi_{\mcO}^{\wedge}(M): G \to \integer[[z]]$ defined by
\begin{equation}
	\chi_{\mcO}^{\wedge}(M)(g) = \sum_{i=0}^\infty \chi\mcO_g(\Lambda^i M)z^i.
	\label{eq:exteriorcriticalseries}
\end{equation}
\end{definition}
We note here that
	it is possible to use tensor or symmetric powers
	instead of (or together with) exterior powers to define different 
	invariants.
We choose to study exterior powers because
	the number of non-trivial powers is
	equal to the number of generators of the module,
	in contrast with
	tensor and symmetric powers, which will be non-zero
	for all positive powers (except when the module is zero).

\subsection*{Conventions for Formal Series}
We denote by $\integer[[G]]$ the collection of
	all functions from $G$ to $\integer$.
The subcollection consisting of functions with finite support
	can be identified with $\integer[G]$,
	the free abelian group generated by $G$.
We embed $G$ as a subset of $\integer[G]$ via the
	\emphb{indeterminate embedding}
	$x:G \to \integer[G]$ defined by
\[
	x(g)(h) =
	\begin{cases}
		1 & ; h = g \\
		0 & ; h \ne g
	\end{cases}
\]
We adopt the more common notation of writing $x^g$ instead of $x(g)$.
The symbol $x$ will be called an \emphb{indeterminate}.
This allows expressing an element of $\integer[G]$ as a finite sum
	$\sum_{i=1}^n n_i x^{g_i}$ where $n_i \in \integer$ and $g_i \in G$.
We extend this convention to allow a formal infinite series
	to represent an element of $\integer[[G]]$.
(This construction remains valid with any unital ring replacing $\integer$.)
With this convention, \eqref{eq:causalseries},
	\eqref{eq:criticalseries} and \eqref{eq:exteriorcriticalseries}
	can be written in full as
\begin{align*}
	\mcH_\mcO(M) & = \sum_{g \in G} \sum_{i=0}^\infty
	\Rank\pars{(H\mcO_g)_i(M)}t^ix^g \\
	\chi_{\mcO}(M) & = \sum_{g \in G} \sum_{i=0}^\infty
	(-1)^i\Rank\pars{(H\mcO_g)_i(M)}x^g \\
	\chi_{\mcO}^\wedge(M) & =
	\sum_{i=0}^\infty \sum_{g \in G} \sum_{j=0}^\infty
	(-1)^j\Rank\pars{(H\mcO_g)_j(\Lambda^iM)}x^g z^i.
\end{align*}

\subsection{Structure Theorems}
\label{section:structuretheorems}
A subset $S \subseteq G$ is \emphb{convex} if
	for all $g, g' \in S$ with $g \preceq g'$,
	$[g, g'] \subseteq S$.
We also define $(-\infty, g']$, $(\infty, g')$, $[g, \infty)$ and $(g, \infty)$
	in an obvious way (assuming $-\infty$ and $\infty$ are not in $G$).

A \emphb{region module of $S \subseteq G$}, written $\Region(S)$,
	is a persistence module defined by
\begin{align*}
	\pars{\Region(S)}_g & =
	\begin{cases}
		R & ; g \in S \\
		0 & ; g \notin S
	\end{cases} \\
	\pars{\Region(S)}_g^x & =
	\begin{cases}
		\text{identity} & ; [g, xg] \subseteq S \\
		0 & ; [g, xg] \not\subseteq S
	\end{cases}.
\end{align*}
A region module $\Region(S)$ is \emphb{indecomposable} if
	$S$ is convex.
An indecomposable region module $M$ have the property
	that $M_g^x$ is always an identity when $M_{xg} \ne 0$.

Since region modules from intervals of the form $[g, g']$,
	$[g, g')$, $(g, g')$ and $(g, g']$
	are of special importance ($g$ may be $-\infty$ and $g'$ may be $\infty$),
	we will omit the outer parentheses and simply write
	$\Region[g, g']$, $\Region[g, g')$ and so on.

The first known structure theorem is
	the classification of finitely generated graded modules
	over a graded principal ideal domain
	\cite{ComputingPersistentHomology}.
This theorem applies
	when $R$ is a field and $X = G = \integer_{\ge 0}$,
	where the action of $X$ on $G$ is the usual addition.
\begin{theorem}
	Let $R$ be a field and $X = G = \integer_{\ge 0}$.
	Then, every finitely generated persistence module $M$
		can be written as
	\[
		M \cong \bigoplus_{i=1}^n \Region[\alpha_i, \alpha_i + \ell_i)
	\]
	where $\alpha_i \in G$, $\ell_i > 0$, and $\ell_i$ may be $\infty$.
\end{theorem}
Webb \cite{DecompositionOfGradedModules} 
	proved a more general result.
$G$ can be extended to cover negative degrees.
\begin{theorem}
	\label{theorem:decompositionofgradedmodules}
	Let $R$ be a field, $X = \integer_{\ge 0}$ and $G = \integer$.
	Then, $M$ can be written as a direct sum of cyclic module
	\[
		M \cong \bigoplus_{i \in \mcI} \Region[\alpha_i, \alpha_i + \ell_i)
	\]
	where $\alpha_i \in G$, $\ell_i > 0$, and $\ell_i$ may be $\infty$,
	if one of the following is true:
	\begin{enumerate}
	\item
		$M$ is \emphb{pointwise finite-dimensional} ($\dim(M_g) < \infty$ for all $g \in G$)
		and $M$ contains no injective submodule.
	\item
		$M$ is \emphb{bounded below} (there exists $g \in G$ such that
			$g' \prec g$ implies $M_g = 0$).
	\end{enumerate}
\end{theorem}
Crawley-Boevey \cite{DecompositionOfPointwiseFiniteDimensionalPersistenceModules}
	gave a different extension
	for locally finite modules.
Assuming $M$ is locally finite,
	$G$ is allowed to be any totally ordered set with
	a countable subset.
The requirements on $X$ can also be weakened
	because in \cite{DecompositionOfPointwiseFiniteDimensionalPersistenceModules},
	$G$ is not assumed to be translation invariant,
	but this is of little significance.
Here is a special case of the result by Crawley-Boevey
	that is more than sufficient for our use.
\begin{theorem}
	\label{theorem:barcodedecomposition}
	Let $R$ be a field, $X$ a submonoid of $\real_{\ge 0}$,
		$G$ a subset of $\real$, and
		$(X, G)$ a persistence grading.
	If $M$ is pointwise finite-dimensional, then
	\[
		M \cong \pars{\bigoplus_{i \in \mcA} \Region(\alpha_i, \alpha'_i)}
		\oplus \pars{\bigoplus_{i \in \mcB} \Region(\beta_i, \beta'_i]}
		\oplus \pars{\bigoplus_{i \in \mcC} \Region[\gamma_i, \gamma'_i]}
		\oplus \pars{\bigoplus_{i \in \mcD} \Region[\delta_i, \delta'_i)}
	\]
	where $\alpha_i < \alpha'_i$,
	$\beta_i < \beta'_i$,
	$\gamma_i \le \gamma'_i$ and
	$\delta_i < \delta'_i$.
	$\alpha_i, \beta_i$ and $\delta_i$ may be $-\infty$, while
	$\alpha'_i, \gamma'_i$ and $\delta'_i$ may be $\infty$.
\end{theorem}

\subsection*{Causality}
A persistence module $M$ is \emphb{causal} if
	for every $g \in G$, there exists $x \in X - \brcs{0}$
	such that $x' \preceq x$ implies $M_g^{x'}: M_g \to M_{x'g}$ is
	an isomorphism.
Regions of the forms $\Region(g, g')$, $\Region(g, g']$ and
	$\Region[g, g']$ are not causal.
If a persistence module $M$ can be decomposed as $M = M' \oplus M''$
	with $M'$ non-causal, then $M$ itself is non-causal.
Combining this with Theorem \ref{theorem:barcodedecomposition},
	we get
\begin{corollary}[Causal Barcode Decomposition]
	\label{corollary:causalbarcodedecomposition}
	Let $R$ be a field, $X$ a submonoid of $\real_{\ge 0}$,
		$G$ a subset of $\real$, and $(X, G)$ a persistence grading.
	If $M$ is pointwise finite-dimensional and causal, then
		it can be written as
	\[
		M \cong \bigoplus_{i \in \mcI}
			\Region[\alpha_i, \alpha_i + \ell_i)
	\]
	where $\alpha_i \in G$, $\ell_i > 0$, and $\ell_i$ may be $\infty$.
\end{corollary}
This covers the first case of Theorem 
	\ref{theorem:decompositionofgradedmodules}
	as well as allowing $G$ and $X$ to be
	different from $\integer$ and $\integer_{\ge 0}$.

\section{Completeness of Exterior Critical Series
	in 1-Dimensional Case}
\label{section:main}
In this section, we assume that $R$ is a field,
	$G$ is a subset of $\real$,
	and $X$ is a submonoid of $\real_{\ge 0}$ acting on $G$ by addition.
(This means $a \in G$ implies $a + X \subseteq G$.)
Also, we assume
	$M$ is a persistence module that is finitely presented
	and bounded,
	hence it is necessarily causal and pointwise finite-dimensional.
Corollary \ref{corollary:causalbarcodedecomposition} gives
	an interval decomposition
\begin{equation}
	M \cong \bigoplus_{i=1}^n \Region[\alpha_i, \alpha_i + \ell_i)
	\label{eq:boundedfinitedecomposition}
\end{equation}
where $\alpha_i \in G$ and $\ell_i > 0$.
(Boundedness forces $\ell_i < \infty$.)
Since there is an embedding
	of $\PersMod(R, X, G)$ into $\PersMod(R, \real_{\ge 0}, \real)$,
	and such embedding preserves finite presentation and boundedness,
	we will assume that
	$X = \real_{\ge 0}$ and $G = \real$.
It should be clear how the main result, to be stated,
	can be applied to more general $G$ and $X$.
For ease of notation, we define $X_+ = \real_{> 0}$.

\begin{definition}
	A \emphb{barcode} of $M$ with the decomposition
		in \eqref{eq:boundedfinitedecomposition}, is
	\begin{equation}
		\varepsilon(M) =
		\sum_{i=1}^n x^{\alpha_i}y^{\ell_i} \in 
			\natural_0\brks{G \times X_+}
		\label{eq:finitebarcode}
	\end{equation}
	where $x$ and $y$ are indeterminates, and
	$\natural_0 = \brcs{0, 1, 2, \ldots}$.
\end{definition}
It is immediate that $\varepsilon(M)$, if exists, determines
	the isomorphism class of $M$.
It is also true that
	an arbitrary member of $\natural_0\brks{G \times X_+}$
	has a corresponding isomorphism class of
	persistence modules.
We will call a member of $\natural_0\brks{G \times X_+}$
	a \emphb{barcode}.

A monomial $b = x^{\alpha}y^{\ell}$ is
	called a \emphb{bar}.
The \emphb{birth grade of $b$} is $\alpha$.
The \emphb{lifespan of $b$} is $\ell$.
The \emphb{death grade of $b$} is $\alpha + \ell$.

A barcode is essentially a multiset of bars.
Suppose $f$ is a barcode
\[
	f = \sum_{i=1}^n x^{\alpha_i}y^{\ell_i}.
\]
We define
\begin{itemize}
\item
	The \emphb{birth series of $f$}:
	\[
		B(f) = \sum_{i=1}^n x^{\alpha_i} \in \natural_0[G].
	\]
\item
	The \emphb{death series of $f$}:
	\[
		D(f) = \sum_{i=1}^n x^{\alpha_i + \ell_i} \in \natural_0[G].
	\]
\item
	The \emphb{critical series of $f$}:
	\[
		C(f) = B(f) - D(f) \in \integer[G].
	\]
\item
	The \emphb{lifespan series of $f$}:
	\[
		L(f) = \sum_{i=1}^n x^{\ell_i} \in \natural_0[X_+].
	\]
\end{itemize}
Note that all these \emph{series} are actually \emph{polynomials}
	(with possibly non-integer exponents).
The maps $f \mapsto B(f)$, $f \mapsto D(f)$, $f \mapsto C(f)$ and
	$f \mapsto L(f)$ are $\integer$-module homomorphisms.

\begin{proposition}
\label{proposition:finitelypresentedmatching}
If $M$ is finitely presented and bounded,
	then $C(\varepsilon(M)) = \chi_{\mcO}(M)$.
\end{proposition}
\begin{proof}
	The decomposition \eqref{eq:boundedfinitedecomposition}
		suggests that
		$M$ has a two-step free resolution
	\[
		0 \to F_1 \stackrel{\psi}{\to} F_0 \stackrel{\phi}{\to} M \to 0
	\]
		such that
	\begin{itemize}
	\item
		$F_0$ is generated by $f_1, \ldots, f_n$;
	\item
		$F_1$ is generated by $h_1, \ldots, h_n$;
	\item
		$\deg(f_i) = \alpha_i$;
	\item
		$\deg(h_i) = \alpha_i + \ell_i$ (where $h_i$ are
		considered elements of $F_0$);
	\item
		$\Ann_X(\phi(f_i)) = [\ell_i, \infty)$.
	\end{itemize}
	Since the resolution has finite length
		and all these modules are finitely generated,
		$M$ admits a critical series.
	From the definition of $\mcO_g$,
		$\dim\pars{(H\mcO_g)_0(M)}$ is equal to the number
		of $i$ such that $\alpha_i = g$, and
		$\dim\pars{(H\mcO_g)_1(M)}$ is equal to the number of
		$i$ such that $\alpha_i + \ell_i = g$.
	This translates to
		$\mcH_{\mcO}(M) = B(\varepsilon(M)) + tD(\varepsilon(M))$
		(recall that $t$ is the indeterminate in $\mcH\mcO_g(M)$, $g \in G$),
		and so $C(\varepsilon(M)) = \chi_\mcO(M)$.
\end{proof}

\subsection{Tensor, Symmetric and Exterior Powers of Barcodes}
Assuming that $M$ can be decomposed as in
	\eqref{eq:boundedfinitedecomposition},
	there exist a set of homogeneous generators
	$g_1, \ldots, g_n$ of $M$ such that
	$\deg(g_i) = \alpha_i$ and
	$\Ann_X(g_i) = [\ell_i, \infty)$.
Each $g_i$ corresponds exactly to the monomial $x^{\alpha_i}y^{\ell_i}$
	in \eqref{eq:boundedfinitedecomposition}.

From this set of generators,
	we obtain a set of generators of
	$M^{\otimes p}$, the $p$-th tensor power of $M$.
There are $n^p$ generators, namely $g_{i_1} \otimes \ldots \otimes g_{i_p}$
	where $i_j \in \brcs{1, \ldots, n}$.
Suppose $g = g_{i_1} \otimes \ldots \otimes g_{i_p}$.
Then $\deg(g) = \sum_{j=1}^p \deg(g_{i_j}) = \sum_{j=1}^p \alpha_{i_j}$,
	and $\Ann_X(g) = \bigcup_{j=1}^p \Ann_X(g_{i_j})
	= [\min_j \ell_{i_j}, \infty)$.
The means the barcode of $M^{\otimes p}$ is
\begin{equation}
	\varepsilon\pars{M^{\otimes p}} =
	\sum_{i_1, \ldots, i_p \in \brcs{1, \ldots, n}}
		x^{\alpha_{i_1} + \ldots + \alpha_{i_p}}
		y^{\min\brcs{\ell_{i_1}, \ldots, \ell_{i_p}}}.
	\label{eq:tensorpowerofbarcode}
\end{equation}
In the same way, we can argue about the barcodes of
	$S^p(M)$ and $\Lambda^p(M)$ and obtain
\begin{align}
	\varepsilon\pars{S^p(M)} & =
	\sum_{1 \le i_1 \le i_2 \le \ldots \le i_p \le n}
		x^{\alpha_{i_1} + \ldots + \alpha_{i_p}}
		y^{\min\brcs{\ell_{i_1}, \ldots, \ell_{i_p}}}.
	\label{eq:symmetricpowerofbarcode}\\
	\varepsilon\pars{\Lambda^p(M)} & =
	\sum_{1 \le i_1 < i_2 < \ldots < i_p \le n}
		x^{\alpha_{i_1} + \ldots + \alpha_{i_p}}
		y^{\min\brcs{\ell_{i_1}, \ldots, \ell_{i_p}}}.
	\label{eq:exteriorpowerofbarcode}
\end{align}
We take \eqref{eq:tensorpowerofbarcode},
	\eqref{eq:symmetricpowerofbarcode} and
	\eqref{eq:exteriorpowerofbarcode} as definitions of
	the \emphb{$p$-th tensor power}, \emphb{symmetric power}
	and \emphb{exterior power} of barcodes:
\begin{align}
	\pars{\sum_{i=1}^n x^{\alpha_i}y^{\ell_i}}^{\otimes p}
	& = 
		\sum_{i_1, \ldots, i_p \in \brcs{1, \ldots, n}}
			x^{\alpha_{i_1} + \ldots + \alpha_{i_p}}
  		y^{\min\brcs{\ell_{i_1}, \ldots, \ell_{i_p}}} \\
	\pars{\sum_{i=1}^n x^{\alpha_i}y^{\ell_i}}^{\odot p}
	& = 
		\sum_{1 \le i_1 \le i_2 \le \ldots \le i_p \le n}
			x^{\alpha_{i_1} + \ldots + \alpha_{i_p}}
			y^{\min\brcs{\ell_{i_1}, \ldots, \ell_{i_p}}} \\
	\pars{\sum_{i=1}^n x^{\alpha_i}y^{\ell_i}}^{\wedge p}
	& = 
		\sum_{1 \le i_1 < i_2 < \ldots < i_p \le n}
			x^{\alpha_{i_1} + \ldots + \alpha_{i_p}}
			y^{\min\brcs{\ell_{i_1}, \ldots, \ell_{i_p}}}.
\end{align}
This makes $\varepsilon$ commute with the tensor power,
	symmetric power and exterior power:
	$\varepsilon(M^{\otimes p}) = \varepsilon(M)^{\otimes p}$,
	$\varepsilon(S^p(M)) = \varepsilon(M)^{\odot p}$ and
	$\varepsilon(\Lambda^p(M)) = \varepsilon(M)^{\wedge p}$.

Following the convention that $\Lambda^0(M) = R$,
	we define $f^{\wedge 0} = 1$ for all
	$f \in \natural_0\brks{G \times X_+} - \brcs{0}$.
Note that $1 \notin \natural_0\brks{G \times X_+}$
	as $\natural_0\brks{G \times X_+}$ is a proper ideal
	of $\natural_0\brks{G \times X}$,
	which has $1$.

For convenience, we also define the \emphb{$p$-th tensor power},
	\emphb{symmetric power} and \emphb{exterior power} of
	members of
	$\natural_0[G]$ by
\begin{align}
	\pars{\sum_{i=1}^n x^{\alpha_i}}^{\otimes p} & =
	\sum_{i_1, \ldots, i_p \in \brcs{1, \ldots, n}}
	x^{\alpha_{i_1} + \alpha_{i_2} + \ldots + \alpha_{i_p}}
	\label{eq:tensorpowerofseries}\\
	\pars{\sum_{i=1}^n x^{\alpha_i}}^{\odot p} & =
	\sum_{1\le i_1 \le i_2 \le \ldots \le i_p\le n}
	x^{\alpha_{i_1} + \alpha_{i_2} + \ldots + \alpha_{i_p}}
	\label{eq:symmetricpowerofseries}\\
	\pars{\sum_{i=1}^n x^{\alpha_i}}^{\wedge p} & =
	\sum_{1\le i_1 < i_2 < \ldots < i_p\le n}
	x^{\alpha_{i_1} + \alpha_{i_2} + \ldots + \alpha_{i_p}}.
	\label{eq:exteriorpowerofseries}
\end{align}

\subsection{Main Result} \label{section:mainresult}
Now that we have the correspondence between
	isomorphism classes of persistence modules
	and their barcodes,
	we will be working solely with polynomials in this section.

The \emphb{exterior critical series of barcode $f$} is
	a member of $\integer\brks{G \times \natural_0}$ defined by
\begin{equation}
	C^\wedge(f) = \sum_{p=1}^\infty C\pars{f^{\wedge p}}z^p
	\label{eq:exteriorcriticalseriesofbarcode}
\end{equation}
where $z$ is the second indeterminate.

Note that the sum is actually finite because
	$f^{\wedge p}$ will be zero for all $p > f|_{(x, y) = (1, 1)}$.
(Recall that the two indeterminates of $f$ are $x$ and $y$.)
The main focus of this paper is to prove
\begin{theorem} \label{theorem:main}
	The map $f \mapsto C^{\wedge}(f)$ is one-to-one.
\end{theorem}
The method of proof is to assume that $C^\wedge(f)$ is given,
	then construct $f$ from $C^\wedge(f)$.
Let us start by assuming that
\begin{equation}
	f = \sum_{i=1}^n x^{\alpha_i}y^{\ell_i}
	\label{eq:barcode}
\end{equation}
with $\ell_1 \le \ell_2 \le \ldots \ell_n$.
Then, for $p \in \natural$,
\begin{align}
	f^{\wedge p} & = \sum_{1 \le i_1 < i_2 < \ldots < i_p \le n}
		x^{\alpha_{i_1} + \ldots + \alpha_{i_p}}
		y^{\ell_{i_1}}
		\label{eq:exteriorpowers} \\
	B\pars{f^{\wedge p}} & = \sum_{1 \le i_1 < i_2 < \ldots < i_p \le n}
		x^{\alpha_{i_1} + \ldots + \alpha_{i_p}}
		\label{eq:birthpowers} \\
	D\pars{f^{\wedge p}} & = \sum_{1 \le i_1 < i_2 < \ldots < i_p \le n}
		x^{\alpha_{i_1} + \ldots + \alpha_{i_p} + \ell_{i_1}}
		\label{eq:deathpowers} \\
	C\pars{f^{\wedge p}} & = \sum_{1 \le i_1 < i_2 < \ldots < i_p \le n}
		x^{\alpha_{i_1} + \ldots + \alpha_{i_p}}\pars{
		1 - x^{\ell_{i_1}}}
		\label{eq:criticalpowers} \\
	L\pars{f^{\wedge p}} & = \sum_{1 \le i_1 < i_2 < \ldots < i_p \le n}
		x^{\ell_{i_1}}.
		\label{eq:lifespanpowers}
\end{align}
The first piece of information we can get from $C^{\wedge}(f)$ is $n$:
	$n$ is the largest integer such that
	$C(f^{\wedge n}) \ne 0$.
\begin{proposition} \label{proposition:numberofbars}
	The number of bars of $f$, which is equal to $n = f|_{(x, y) = (1, 1)}$,
		can be determined from $C^\wedge(f)$.
\end{proposition}
\begin{proof}
	From \eqref{eq:criticalpowers},
		$n$ is the unique non-negative integer such that
		$C(f^{\wedge n}) \ne 0$ and
		$C(f^{\wedge (n + 1)}) = 0$.
\end{proof}

\subsubsection{Moment} \label{section:moment}
For each $P \in \real[G]$, we define the \emphb{moment of $P$} by
\[
	\mu(P) = P'|_{x=1} \in \real
\]
where $P'$ is the formal derivative of $P$, and
	$P'|_{x=1}$ is the result from substitute $1$ into $x$ in $P'$.
A more explicit definition is
\[
	\mu\pars{\sum_{i=1}^n r_i x^{\alpha_i}}
	= \sum_{i=1}^n r_i \alpha_i.
\]
$\mu$ is obviously an $\real$-module homomorphism from $\real[G]$
	(or any of its subrings)
	to $\real$.

\subsubsection{Lifespan Series}
If $f$ is defined by \eqref{eq:barcode}, then
\[
	\mu(C(f)) = \mu(B(f)) - \mu(D(f))
	= \sum_{i=1}^n \alpha_i - \sum_{i=1}^n (\alpha_i + \ell_i)
	= -\sum_{i=1}^n \ell_i.
\]
By a counting argument and the condition that
	$\ell_1 \le \ell_2 \le \ldots \le \ell_n$,
	we get
\begin{equation}
	\mu(C(f^{\wedge p})) =
	-\sum_{i=1}^{n-p+1}
	\binom{n-i}{p-1} \ell_i =
	-\ell_{n-p+1} - \sum_{i=1}^{n-p}\binom{n-i}{p-1} \ell_i.
	\label{eq:measures}
\end{equation}
for $p \in \brcs{1, 2, \ldots, n}$.
Assuming that $C(f^{\wedge p})$ are given for all $p$, we can compute
	all $\mu(C^{\wedge p})$, and
	\eqref{eq:measures} is a system of $n$ linear equations
	with $\ell_i$ as knowns.
The solution to the system is given by the recursive formula
\[
	\ell_{i} = -\mu(C(f^{\wedge (n - i + 1)})) -
		\sum_{j=1}^{i - 1} \binom{n - j}{n - i}\ell_j.
\]
This means $L(f)$ can be found from $C^{\wedge}(f)$.

It is easy to see from \eqref{eq:lifespanpowers} that
	once $L(f)$ is known, $L(f^{\wedge p})$ can be computed
	for any $p \in \natural$.
Thus, we have proved
\begin{proposition} \label{proposition:lifespanseries}
	$L(f)$, hence all $L(f^p)$, can be determined from
		$C^\wedge(f)$.
\end{proposition}

\subsubsection{Drift}
With $f$ in \eqref{eq:barcode},
	define the \emphb{drift of $f$} as
\[
	\Delta(f) = \sum_{i=1}^n \alpha_i
	= \mu\pars{f|_{y=1}}.
\]
The drift can be computed by substituting $y = 1$
	then computing the moment, so
	$\Delta$ is an $\real$-module homomorphism.
Direct computation yields
\begin{equation}
	\Delta(f^{\wedge p}) = \binom{n - 1}{p - 1}\Delta(f).
	\label{eq:driftofpower}
\end{equation}
When $p = n$, we get $\Delta(f^{\wedge n}) = \Delta(f)$.
However, $f^{\wedge n}$ has only one bar:
	$f^{\wedge n} = x^{\Delta(f)}y^{\ell_1}$.
We see that
	$C(f^{\wedge n}) = x^{\Delta(f)} - x^{\Delta(f) + \ell_1}$,
	so
\begin{align}
	x^{\Delta(f)} & = \frac{1}{1 - x^{\ell_1}}C(f^{\wedge n}) \nonumber\\
	\therefore \quad
	\Delta(f) & = \mu\pars{\frac{1}{1 - x^{\ell_1}}C(f^{\wedge n})},
	\label{eq:drift}
\end{align}
where the last equation comes from computing the moment on both sides.
This, combined with \eqref{eq:driftofpower}, proves
\begin{proposition} \label{proposition:drift}
	$\Delta(f^{\wedge p})$,
		for any $p \in \natural$, can be obtained from
		$n$ and $C(f^{\wedge n})$.
\end{proposition}

\subsubsection{The Case of Single Lifespan}
If the lifespan series $L(f)$ consists
	of only one monomial, \ie, all bars in $f$ have the same lifespan,
	then $L(f) = nx^\ell$, $f = \sum_{i=1}^n x^{\alpha_i}y^{\ell}$,
	and $D(f) = x^\ell B(f)$.
Since $C(f) = B(f) - D(f) = \pars{1 - x^\ell}B(f)$,
\[
  B(f) = \frac{C(f)}{1 - x^\ell}.
\]
Also, $f = \sum_{i=1}^n x^{\alpha_i}y^\ell = y^\ell B(f)$, hence
\begin{equation}
	f = \frac{y^\ell}{1 - x^\ell}C(f). \label{eq:singlelifespan}
\end{equation}
We have shown that
\begin{proposition} \label{proposition:singlelifespan}
	The map $f \mapsto C(f)$ is one-to-one if the domain is restricted to
	\[
		\brcs{\sum_{i=1}^n x^{\alpha_i}y^{\ell}}
	\]
		with $\ell$ fixed.
\end{proposition}

\subsubsection{The Case of Two Lifespans with One Outlier}
If we know beforehand that in a barcode $f$,
	one bar has lifespan of $\tilde\ell$ and all other bars
	have the same lifespan of $\ell \ne \tilde\ell$,
	it is possible to extract $f$ by using
	$C(f)$ and $\Delta(f)$.
We demonstrate the procedure below.

Suppose $\ell_j = \tilde\ell$ ($j$ fixed) and $\ell_i = \ell$
	for $i \ne j$.
(According to our ordering $\ell_1 \le \ell_2 \le \ldots \ell_n$,
	we must have either $j = 1$ or $j = n$.)
We can determine $\alpha_j$ as follows:
\begin{align*}
	f & = x^{\alpha_j}\pars{y^{\tilde\ell} - y^{\ell}} +
		\sum_{i=1}^n x^{\alpha_i}y^{\ell} \\
	C(f) & =
		\pars{x^{\alpha_j + \ell} - x^{\alpha_j + \tilde\ell}} +
		\sum_{i=1}^n \pars{x^{\alpha_i} - x^{\alpha_i + \ell}} \\
	xC(f)' & =
		\pars{(\alpha_j + \ell)x^{\alpha_j + \ell}
			- (\alpha_j + \tilde\ell)x^{\alpha_j + \tilde\ell}} +
		\sum_{i=1}^n \pars{\alpha_i x^{\alpha_i} -
		(\alpha_i + \ell)x^{\alpha_i + \ell}} \\
	\mu(xC(f)') & =
		\pars{(\alpha_j + \ell)^2 - (\alpha_j + \tilde\ell)^2} +
		\sum_{i=1}^n \pars{\alpha_i^2 - (\alpha_i + \ell)^2} \\
	& = (2\alpha_j + \ell + \tilde\ell)
	(\ell - \tilde\ell) - \ell\sum_{i=1}^n (2\alpha_i + \ell)\\
	& = 2\alpha_j(\ell - \tilde\ell) +
	(\ell + \tilde\ell)(\ell - \tilde\ell)
	- 2\ell\Delta(f) - n\ell^2 \\
	\therefore \quad
	\alpha_j & = \frac{1}{2(\ell - \tilde\ell)}
	\pars{
		\mu(xC(f)') - (\ell + \tilde\ell)(\ell - \tilde\ell) + 2\ell\Delta(f)
		+ n\ell^2
	}.
\end{align*}
This means we can compute $\alpha_j$ from $\tilde\ell$, $\ell$, $C(f)$
	and $\Delta(f)$.
Next, let
\[
	\tilde f = f - x^{\alpha_j}y^{\tilde\ell}
	= \sum_{\substack{1\le i \le n\\i \ne j}} x^{\alpha_i}y^\ell.
\]
$\tilde f$ has $n - 1$ bars, all of which have equal lifespan
	of $\ell$, and
	$C(\tilde f) = C(f) - C(x^{\alpha_j}y^{\ell_j})
	= C(f) - \pars{x^{\alpha_j} - x^{\alpha_j + \tilde\ell}}$.
Hence, we can use Proposition \ref{proposition:singlelifespan}
	(equation \eqref{eq:singlelifespan})
	to get $\tilde f$ from $C(\tilde f)$,
	and finally obtain $f$ from $f = \tilde f + x^{\alpha_j}y^{\tilde\ell}$.
This result, combined with Proposition \ref{proposition:singlelifespan}, gives
\begin{proposition} \label{proposition:twolifespans}
	The map $f \mapsto C(f)$ is one-to-one if the domain is restricted to
	\[
	\brcscond{f = x^{\tilde\alpha}y^{\tilde\ell}
	+ \sum_{i=1}^n x^{\alpha_i}y^{\ell_i}}{\Delta(f) = \delta}
	\]
	with $\ell, \tilde\ell$ and $\delta$ fixed.
\end{proposition}

\subsubsection{Extracting Birth and Death Series}
Suppose we are given $C^{\wedge}(f)$,
	where $f$ is of the form \eqref{eq:barcode}.
Consider the $(n-1)$-th exterior power of $f$:
\begin{align*}
	f^{\wedge(n - 1)} & =
		x^{\Delta(f) - \alpha_1}y^{\ell_2} +
		\sum_{i=2}^n x^{\Delta(f) - \alpha_i}y^{\ell_1}.
\end{align*}
We see that one bar in $f^{\wedge(n-1)}$ has
	lifespan of $\ell_2$ while all other bars have lifespan of $\ell_1$.
(All $\ell_i$ are known by
	Proposition \ref{proposition:lifespanseries}.)
$\Delta(f^{\wedge(n-1)})$ is known by
	Proposition \ref{proposition:drift}.
Hence, Proposition \ref{proposition:twolifespans} applies to give us
	$f^{\wedge(n - 1)}$.
We can then compute $B(f^{\wedge(n-1)})$ and relate it to $B(f)$:
\[
	B(f^{\wedge(n-1)}) = \sum_{i=1}^n x^{\Delta(f) - \alpha_i}
	= x^{\Delta(f)}\sum_{i=1}^n x^{-\alpha_i}.
\]
Replacing $x$ by $1/x$, we get
\[
	B(f) = x^{\Delta(f)}B(f^{\wedge(n-1)})|_{1/x}
\]
where $B(f^{\wedge(n-1)})|_{1/x}$ is the result from
	replacing $x$ by $1/x$ in $B(f^{\wedge(n-1)})$.
Now that $B(f)$ has been determined,
	\eqref{eq:exteriorpowerofseries} and
	\eqref {eq:birthpowers} give
	$B(f^{\wedge p}) = B(f)^{\wedge p}$ for all $p \in \natural$.
Then, $D(f^{\wedge p})$ can be obtained from
	$D(f^{\wedge p}) = C(f^{\wedge p}) + B(f^{\wedge p})$.
We summarize these results as
\begin{proposition} \label{proposition:birthseries}
	For any $p \in \natural$,
		$B(f^{\wedge p})$ and $D(f^{\wedge p})$ can be obtained
		from $C^{\wedge}(f)$.
\end{proposition}

\subsubsection{Proof of Theorem \ref{theorem:main}}
Suppose $f$ is of the form \eqref{eq:barcode} but
	we are given just $C^{\wedge}(f)$.
We know $n$ (the number of bars in $f$)
	from Proposition \ref{proposition:numberofbars}.
From Proposition \ref{proposition:lifespanseries},
	we know $L(f)$, and so all $\ell_i$.
We will prove the theorem by induction
	on the number of distinct lifespans.

The base case is when
	$\ell_1 = \ell_2 = \ell_3 = \ldots = \ell_n$.
Proposition \ref{proposition:singlelifespan}
	gives $f$ and we are done.

Otherwise,
	there exists $m \in \brcs{1, 2, \ldots, n - 1}$ such that
	$\ell_1 = \ell_2 = \ldots = \ell_m < \ell_{m+1}
	\le \ldots \le \ell_n$.
Let $\ell = \ell_1 = \ldots = \ell_m$ and
	$\tilde f = \sum_{i=m+1}^n x^{\alpha_i}y^{\ell_i}$.
Consider $f^{\wedge(n-m)}$.
Exactly one bar in $f^{\wedge(n-m)}$ has lifespan $\ell_{m+1}$
	while all other bars have lifespan $\ell$.
$\Delta(f^{\wedge(n-m)})$ is known from
	Proposition \ref{proposition:drift}, hence
	we know $f^{\wedge(n-m)}$ completely
	by Proposition \ref{proposition:twolifespans}.
Since
\begin{align*}
	f^{\wedge(n-m)} & =
		x^{\alpha_{m+1} + \ldots + \alpha_n}\pars{y^{\ell_{m+1}} - y^{\ell}} +
		\sum_{1\le i_1 < i_2 < \ldots < i_{n-m}\le n}
		x^{\alpha_{i_1} + \alpha_{i_2} + \ldots + \alpha_{i_{n-m}}}
		y^{\ell} \\
	& = x^{\alpha_{m+1} + \ldots + \alpha_n}\pars{y^{\ell_{m+1}} - y^{\ell}} +
		y^\ell B(f)^{\wedge(m-n)},
\end{align*}
we can solve for $\delta = \alpha_{m+1} + \ldots + \alpha_n$:
\begin{align*}
	x^{\delta} & =
	\frac{1}{y^{\ell_{m+1}} - y^{\ell}}
	\pars{f^{\wedge(n - m)} - y^\ell B(f)^{\wedge(m - n)}} \\
	\delta & = \mu\pars{
		\frac{1}{y^{\ell_{m+1}} - y^{\ell}}
		\pars{f^{\wedge(n - m)} - y^\ell B(f)^{\wedge(m - n)}}
	}.
\end{align*}
Next, consider $f^{\wedge(n-m-1)}$:
\begin{eqnarray*}
	f^{\wedge(n-m-1)} & = &
		x^{\delta - \alpha_{m+1}} \pars{y^{\ell_{m+2}} - y^{\ell}}
		+ \sum_{i=m+2}^{n}
		x^{\delta - \alpha_i}\pars{y^{\ell_{m+1}} - y^{\ell}}
		+ B(f)^{\wedge(n-m-1)}y^{\ell}
\end{eqnarray*}
\begin{equation}
	\therefore \quad C(f^{\wedge(n-m-1)}) =
		x^{\delta - \alpha_{m+1}} \pars{x^{\ell} - x^{\ell_{m+2}}}
		+ \sum_{i=m+2}^{n}
		x^{\delta - \alpha_i}\pars{x^{\ell} - x^{\ell_{m+1}}}
		+ B(f)^{\wedge(n-m-1)}\pars{1 - x^{\ell}}.
	\label{eq:n-m-1barcode}
\end{equation}
Let $g = x^{\delta-\alpha_{m+1}}y^{\ell_{m+2}-\ell}
	+ \sum_{i=m+2}^n x^{\delta-\alpha_i}y^{\ell_{m+1}-\ell}$.
$g$ is simply the result from shortening all bars in $f^{\wedge(n-m-1)}$
	by $\ell$
	and removing those with zero remaining lifespan.
We see that
\begin{eqnarray}
	\Delta(g) & = &
		\sum_{i=m+1}^n \pars{\delta - \alpha_i}
		\ = \ (n-m-1)\delta, \label{eq:deltaofg} \\
	C(g) & = &
		x^{\delta - \alpha_{m+1}}\pars{1 - x^{\ell_{m+2} - \ell}}
		+ \sum_{i=m+2}^n x^{\delta-\alpha_i}\pars{1 - x^{\ell_{m+1} - \ell}}
		\nonumber\\
	& = & x^{-\ell}\pars{
		C(f^{\wedge(n-m-1)}) - B(f)^{\wedge(n-m-1)}\pars{1 - x^\ell}}
		 \label{eq:criticalofg}
\end{eqnarray}
where the last equation comes from \eqref{eq:n-m-1barcode}.
Using Proposition \ref{proposition:twolifespans},
	we can determine $g$ from known information, then compute
\[
	B(g) = x^{\delta}\sum_{i=m+1}^n x^{-\alpha_i} \quad\text{ and }\quad
	\bar B = B(f) - x^{\delta}\pars{B(g)|_{1/x}} = \sum_{i=1}^m x^{\alpha_i},
\]
where $B(g)|_{1/x}$ is the result from substituting $1/x$ into $x$ in $B(g)$.

Let $\tilde f = \sum_{i=m+1}^n x^{\alpha_i}y^{\ell_i}$.
$f$ can be expressed simply as
\begin{equation}
	f = \bar B y^{\ell} + \tilde f.
	\label{eq:induction}
\end{equation}
It follows that for any $p \in \brcs{1, 2, \ldots, n}$,
\begin{align}
	f^{\wedge p} & =
		\tilde f^{\wedge p}
		+ y^{\ell}\sum_{i=1}^p \bar B^{\wedge i} B(f)^{\wedge(p-i)}
		\nonumber\\
	B(f^{\wedge p}) & =
		B(\tilde f^{\wedge p}) +
		\sum_{i=1}^p \bar B^{\wedge i} B(f)^{\wedge(p-i)}
		\label{eq:subbirthseries}\\
	D(f^{\wedge p}) & =
		D(\tilde f^{\wedge p}) +
		x^{\ell}\sum_{i=1}^p \bar B^{\wedge i} B(f)^{\wedge(p-i)}.
		\label{eq:subdeathseries}
\end{align}
Subtract \eqref{eq:subbirthseries} from \eqref{eq:subdeathseries}
	and rearrange to get
\[
	C(\tilde f^{\wedge p}) =
	C(f^{\wedge p}) - \pars{1 - x^\ell}
		\sum_{i=1}^p \bar B^{\wedge i}B(f)^{\wedge(p-i)}.
\]
All terms on the right-hand side of this equation are known,
	so $C^\wedge(\tilde f)$ is known.
By definition, $\tilde f$
	has one fewer distinct lifespan than $f$,
	so $\tilde f$ can be constructed by the induction hypothesis.
$f$ can then be obtained from \eqref{eq:induction}.
This completes the proof.

\section{Conclusions and Future Work}
Since we have already established that
	$C^{\wedge}(\varepsilon(M)) = \chi_{\mcO}^\wedge(M)$
	in Proposition \ref{proposition:finitelypresentedmatching},
	Theorem \ref{theorem:main} translates into the language of persistence
	modules as
\begin{corollary} \label{corollary:main}
	For finitely presented and bounded persistence modules
		in $\PersMod(R, \real_{\ge 0}, \real)$,
		the exterior critical series
		$\chi_{\mcO}(M)$ completely determines the isomorphism class of $M$.
\end{corollary}
The definition of the critical series, by construction,
	is available in a very general setting.
Corollary \ref{corollary:main} shows that
	it carries enough information
	to be complete in the one-dimensional case.
These properties are not different from the rank invariant
	\cite{TheoryOfMultidimensionalPersistence},
	so we note here that
	there are situations where the rank invariant
	cannot distinguish between persistence modules
	that have different exterior critical series.
(See Figure \ref{figure:invariantexample}.)

\begin{figure}
	\begin{center}
		\includegraphics{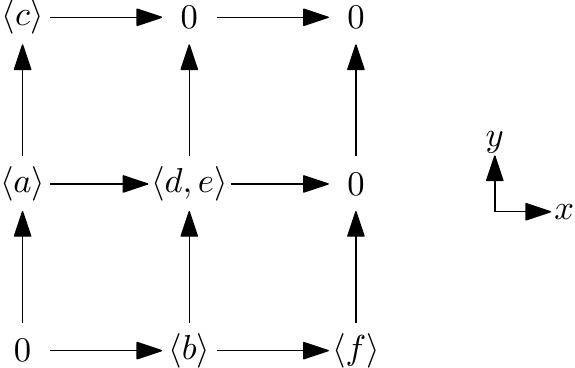}
	\end{center}
	\caption{
		\label{figure:invariantexample}
		An example situation where the exterior critical series
			captures information that is not detected
			by the rank invariant.
		$a, b, c, d, e, f$ are generators
			of a graded vector space over a field $R$.
		We build two persistence modules $M$ and $M'$
			over this graded vector space
			by defining actions of $x$ and $y$ as follows:
		(1) $M$: $ya = c$, $xa = d = yb$ and $xb = f$;
		(2) $M'$: $ya = c$, $xa = d$, $yb = e$ and $xb = f$.
		The rank invariant is the same for both modules
			but the exterior critical series are different, as can be
			seen from the fact that $M$ is generated by at least $3$ elements
			while $M'$ can be generated by $2$ elements.
	}
\end{figure}

We hope that
	Corollary \ref{corollary:main} should extend to
	cover more general cases, such as unbounded tame modules,
	without much difficulty.
Similar results might be available
	for special modules such as zigzag persistence modules
	\cite{ZigzagPersistence}
	embedded as 2-dimensional persistence modules.
The general construction of the $\mcF$-homology sequence
	may allow more information to be detected; for example,
	one may define reverse-onset functors that capture intervals
	of the form $(g, g']$.
Also, incorporating tensor powers and symmetric powers may give information
	that is unavailable with exterior powers alone
	in the case of multi-dimensional persistence.

There is still much to study
	about the computational aspect of this concept.
In its raw form,
	the exterior critical series
	requires $O(2^n)$ storage for a
	persistence module with $n$ generators
	(with persistence dimension $1$), so
	it is not computationally practical.
We do hope, however, that there will be a procedure
	to choose and compute only parts of the exterior critical series
	that provide sufficient information about major characteristics
	of the module.

\bibliographystyle{plain}
\bibliography{persistencehomology}

\end{document}